\documentclass[12pt,twoside]{amsart}
\usepackage[margin=3cm]{geometry}
\usepackage[colorlinks=false]{hyperref}
\usepackage[english]{babel}
\usepackage{graphicx}
\usepackage{float}
\usepackage{amsmath,amsfonts,amssymb,amsthm}
\usepackage{lipsum}
\usepackage[T1]{fontenc}
\usepackage{fourier}
\usepackage{color}
\usepackage{esint}
\usepackage{caption}
\usepackage{ccicons}
\makeatletter
\def\blfootnote{\xdef\@thefnmark{}\@footnotetext}
\makeatother

\newcommand\ccnote{
    \blfootnote{The research is supported by the National Key R and D Program of China 2020YFA0713100, NSFC No 11721101 and NSFC No.12301077.}
}
\allowdisplaybreaks
\usepackage[export]{adjustbox}
\numberwithin{equation}{section}
\usepackage{setspace}
\renewcommand{\le}{\leqslant}
\renewcommand{\leq}{\leqslant}
\renewcommand{\ge}{\geqslant}
\renewcommand{\geq}{\geqslant}
\renewcommand{\mathbb}{\varmathbb}
\usepackage{fancyhdr}
\newtheorem{theorem}{Theorem}[section]
\newtheorem{lemma}[theorem]{Lemma}
\newtheorem{corollary}[theorem]{Corollary}

\newtheorem{remark}[theorem]{Remark}

\newtheorem{pro}[theorem]{Proposition}

\newtheorem{defi}[theorem]{Definition}

\newtheorem{conj}[theorem]{Conjecture}
\usepackage{comment}
\usepackage{enumerate}
\usepackage{hyperref}
\usepackage[alphabetic, msc-links, backrefs]{amsrefs}

\usepackage{amsthm}

\allowdisplaybreaks

\input epsf
\def\begfig {
\begin{figure}
\small }
\def\endfig {
\normalsize
\end{figure}
}

\begin{document}
\thispagestyle{empty}

\ccnote

\title[]{Optimal  rigidity estimates for varifolds  almost minimizing the Willmore energy}

\author{Yuchen Bi}
\author{Jie Zhou}

\address{Yuchen Bi\\Beijing International Center for Mathematical Research, Peking University, Beijing, 100871, P. R. China}
\email{ycbi@bicmr.pku.edu.cn}
\address{Jie Zhou\\ School of Mathematical Sciences\\ Capital Normal University\\ 105 West Third Ring Road North\\ Haidian District\\ Beijing\\ 100048 P.R. China }
\email{zhoujiemath@cnu.edu.cn}
\maketitle

\noindent \textbf{Abstract.} \textit{For an integral  $2$-varifold $V=\underline{v}(\Sigma,\theta_{\ge 1})$ in $\mathbb{R}^n$ with generalized mean curvature $H\in L^2$ such that $\mu(\mathbb{R}^n)=4\pi$ and $\int_{\Sigma}|H|^2d\mu\le 16\pi(1+\delta^2)$ , we show that $\Sigma$ is $W^{2,2}$ close to the standard embedding of the round sphere in a quantitative way when $\delta< \delta_0\ll 1$. For $n=3$, we prove that the sharp constant is  $\delta_0^2=2\pi$.}
\vskip0.3cm

\noindent \textbf{Keywords.} Willmore energy, Varifold,Quantitative rigidity \vspace{0.5cm}

\noindent \textbf{MSC2020.} 49Q20, 35B65


\tableofcontents
\section{Introduction}
Let $V=\underline{v}(\Sigma,\theta)$ be an integral $2$-varifold in $\mathbb{R}^n$. Its generalized mean curvature is defined by a vector valued function $H\in L^1_{loc}(d\mu)$ satisfying the first variation formula
\begin{align*}
\int_{\Sigma}div^{\Sigma}Xd\mu=-\int_{\Sigma}H\cdot X, \forall X\in C^1_c(\mathbb{R}^n),
\end{align*}
where $\mu=\theta\mathcal{H}^2\llcorner \Sigma$ is the volume measure and $div^\Sigma X(x)=tr_{T_x\Sigma}DX$ is the divergence of $X$ with respect to the approximating tangent space $T_x\Sigma$ of the rectifiable varifold. If $H\in L^2(d\mu)$, then the Willmore energy of the varifold is defined by 
\begin{align*}
\mathcal{W}(V)=\frac{1}{4}\int_{\Sigma}|H|^2d\mu.
\end{align*}  
The integrality of $V$ implies the generalized mean curvature is perpendicular to the approximating tangent space of the varifold\cite{B-1978} and hence there holds the monotonicity formula\cites{LS93, KS}, which further implies the Li-Yau inequality and the rigidity of Willmore minimizer in the class of integral varifolds with compact support\cite{LY-1982, LS-2014}. That is,  there always holds 
$$\mathcal{W}(V)\ge 4\pi,$$
and the equality holds if and only if $V=\underline{v}(\mathbb{S}^2,1)$ is a multiplicity one round sphere.  

A natural question is the quantitative rigidity estimate for almost Willmore minimizer. 

In the smooth setting, the question is equivalent to the quantitative rigidity estimate of nearly umbilical surfaces and the optimal rigidity estimates has been established by  De Lellis-M\"uller\cite{dLMu, dLMu2} and Lamm-Sch\"atzle\cite{LS-2014}.  More precisely, for a closed surface embedded in $\mathbb{R}^n$ with genus $g_\Sigma$, there holds 
\begin{align*}
4\pi\le \mathcal{W}(\Sigma)=\frac{1}{4}\int_{\Sigma}|H|^2d\mu=\frac{1}{2}\int_{\Sigma}|\mathring{A}|^2d\mu+4\pi(1-g_\Sigma),
\end{align*}
where $\mathring{A}=A-\frac{H}{2}\otimes g$ is the trace free part of the second fundamental form of $\Sigma$.  So, $\int_{\Sigma}| \mathring{A}|^2d\mu\le 2\delta^2$ implies  $\mathcal{W}(\Sigma)\le 4\pi+ \delta^2$.  On the other hand, since the large genus limit of the infimum of the Willmore energy is $8\pi$\cite{KLS}, by the existence of Willmore minimizer for oriented surfaces with a fixed genus\cite{LS93,BK},  there exists a gap $\delta_0=\delta_0(n)>0$  such that each oriented surface $\Sigma$ with $\mathcal{W}(\Sigma)\le 4\pi+\delta_0^2$  has genus $g_\Sigma=0$. As a conclusion, for an oriented surface $\Sigma$ and $\delta\le \delta_0$, there holds
\begin{align}\label{Willmore and tracefree part}
\mathcal{W}(\Sigma)\le 4\pi +\delta^2 \text{ if and only if } \int_{\Sigma}|\mathring{A}|^2d\mu\le 2\delta^2,
\end{align}
and $\Sigma$ is a topological sphere under either assumption.  

In \cite{dLMu}\cite{dLMu2}, De Lellis and M\"uller obtained  the optimal rigidity estimates for nearly umbilical surfaces in $\mathbb{R}^3$. In \cite{LS-2014}, Lamm and Sch\"atzle  apply the monotonicity formula to  studied this problem and extended the results to arbitrary codimension. We summarize their works \cite{dLMu, dLMu2,LS-2014} as the following. 
\begin{theorem}[$\textbf{De Lellis--M\"{u}ller(2005,2006); Lamm--Sch\"{a}tzle(2014)}$]\label{quantitative umbilical} Assume $\Sigma\subset \mathbb{R}^n$ is a smoothly embedded topological sphere and $\mathcal{H}^2(\Sigma)=4\pi$. If 
$\|\mathring{A}_{\Sigma}\|^2_{L^2(\Sigma)}<2e_n$, where,
\begin{equation*}
    e_n:=\left\{
\begin{aligned}
4\pi &\ \text{ for } n=3, \\
\frac{8\pi}{3} &\ \text{ for } n=4, \\
2\pi &\ \text{ for } n\ge 5,
\end{aligned}
\right.
\end{equation*}
is defined as in  \cite{Sch13},
 then there exists a conformal parametrization $f:\mathbb{S}^2\to \Sigma$ such that $f^*g_{\mathbb{R}^n}=e^{2u}g_{\mathbb{S}^2}$  and
\begin{align}\label{optimal rigidity estimate}
\|f-\text{id}_{\mathbb{S}^2}\|_{W^{2,2}(\mathbb{S}^2)}+\|u\|_{L^\infty(\mathbb{S}^2)}\le C(n,\tau)\|\mathring{A}_\Sigma\|_{L^2(\Sigma)},
\end{align}
where $\tau:=2e_n-\|\mathring{A}_\Sigma\|^2_{L^2(\Sigma)}>0$.
\end{theorem}
By \eqref{Willmore and tracefree part}, this theorem implies the optimal rigidity for smooth surfaces which almost minimize the Willmore energy.

In this paper, we focus on the varifold setting. In this situation, the second fundamental form is not well-defined  a priori, so one cannot directly apply the work of De Lellis-M\"uller or Lamm-Sch\"{a}tzle. However, in the previous work \cite{BZ-2022b}, the authors observed that integral varifolds with critical Allard conditions are regular enough to have a well-defined second fundamental form. Actually, this kind of varifolds admit $W^{2,2}$ conformal parameterizations with bounded conformal factors. The precise statement is:

\begin{theorem}[\cite{BZ-2022b}]\label{2022b-main}
     Let $V=\underline{v}(\Sigma,\theta)$ be an integral $2$-varifold in $U\supset B_1$ with  corresponding Radon measure $\mu=\theta\mathcal{H}^2\llcorner \Sigma$. Assume  $\theta(x)\ge 1$ for $\mu$-a.e. $x\in U$  and $0\in\Sigma=spt\mu$. If $V$ admits generalized mean curvature $H\in L^2(d\mu)$ such that
     \begin{align}\label{density condition}
     \mu(B_1)\le (1+\varepsilon)\pi
     \end{align}
     and
     \begin{equation}\label{mean curvature bound}
          \left(\int_{B_1}|H|^2d\mu\right)^{\frac{1}{2}}\le \varepsilon,
     \end{equation}
     then there exists a bi-Lipschitz conformal parameterization $f: D_1\to f(D_1)\subset \Sigma$ satisfying
     \begin{enumerate}
      \item  $B(0,1-\psi)\cap\Sigma\subset f(D_1)$;
     \item For any $x,y\in D_1$,
     \begin{align*}
     (1-\psi)|x-y|\le |f(x)-f(y)|\le (1+\psi)|x-y|;
     \end{align*}
     \item Let $g=:df\otimes df$, then there exist $w\in W^{1,2}(D_1)\cap L^\infty(D_1)$ such that
             \[g=e^{2w}(dx\otimes dx+dy\otimes dy)\]
            and $$\|w\|_{L^\infty(D_1)}+\|\nabla w\|_{L^2(D_1)} + \|\nabla^2 w\|_{L^1(D_1)}\le \psi;$$
     \item $f\in W^{2,2}(D_1,\mathbb{R}^n)$ and
     \[\|f-\emph{i}_1\|_{W^{2,2}(D_1)}\leq \psi ,\]
   where $\emph{i}_1: D_1\rightarrow \mathbb{R}^n$ is a standard isometric embedding.

     \end{enumerate}
     Here $\psi=\psi(\varepsilon)$ is a positive function such that $\lim_{\varepsilon\to 0}\psi(\varepsilon)=0$.
    \end{theorem}

In view of Theorem \ref{2022b-main} and the work of De Lellis-M\"{u}ller and Lamm-Sch\"{a}tzle, even in the varifold setting, it is also reasonable to expect an optimal rigidity result with only Willmore energy involved in the statement. Here is the main result of this paper:

\begin{theorem}\label{main theorem} Let $V=\underline{v}(\Sigma,\theta)$ be an integral varifold with mass $\mu(\mathbb{R}^n)=4\pi$, $\theta\geq 1$ a.e. and the generalized mean curvature $H\in L^2(d\mu)$ such that 
$$\frac{1}{4}\int_{\mathbb{R}^n}|H|^2d\mu\le 4\pi+\delta^2.$$
Then there exists $\delta_0(n)>0$ such that for any $\delta<\delta_0$,  there exists a round sphere $\mathbb{S}^2$ with radius equal to one and centered at $0\in \mathbb{R}^n$ such that after translation, there exists a conformal parameterization $\Phi: \mathbb{S}^2\to \Sigma$ such that $d\Phi\otimes d\Phi=e^{2v}g_{\mathbb{S}^2}$ and 
\begin{align}\label {estimate}
\sup_{q\in \mathbb{S}^2}\|\Phi-id_{\mathbb{S}^2}\|_{W^{2,2}(\mathbb{S}^2)}+\|v\|_{C^0(\mathbb{S}^2)}\le C\delta. 
\end{align}
\end{theorem}

  It is also a natural question to ask what is the sharp value of $\delta_0^2(n)$ in the varifold setting. The example of double bubble\cite{HMRR02} shows $\delta_0^2(n)\le 2\pi$. In the case $n=3$, we can show $\delta_0^2(3)=2\pi$ (see Corollary \ref{corollary}). In cases $n\ge 4$, the question is more complicated. We can show $\delta_0^2(n)<2\pi$ implies $V$ is regular (See Theorem \ref{sharp constant}). But even for a smooth $V$, this question is related to Kusner's conjecture.

\begin{conj}\cite{K96}
A smooth closed surface in $\mathbb{R}^4$ with Willmore energy smaller than $6\pi$ has to be a sphere.
\end{conj}

   Next, let us introduce our strategies  and the structure of this paper. In Section \ref{sec:Pre} , we prepare some preliminaries for varifolds and introduce the inversion formula and  apply the inversion formula to relate the varifold almost minimizing the Willmore energy to a new varifold satisfying the critical Allard condition \eqref{density condition} and \eqref{mean curvature bound} on the whole $\mathbb{R}^n$. In Section \ref{sec:Quantitative plane}, we  apply Theorem \ref{2022b-main} to  get the regularity of  varifolds in $\mathbb{R}^n$ satisfying the critical Allard condition and obtain its optimal quantitative rigidity estimate(See Theorem \ref{inverted quantitative rigidity}). 
   In Section \ref{sec:Quantitative sphere}.
    we invert the quantitative rigidity of the plane back and get the quantitative rigidity of varifolds almost minimizing the Willmore energy. 
   In the last Section \ref{sec:sharp constant}, we discuss the sharp constants.

\section{Preliminaries}\label{sec:Pre}

In this paper, we always consider rectifiable $2$-varifold $V=\underline{v}(\Sigma, \theta)$ in $\mathbb{R}^n$ with $\theta\ge 1$. For any such varifold,  put $\mu=\mu_V=\theta\mathcal{H}^2\llcorner \Sigma$.  We also assume the generalized mean curvature $H\in L^2(d\mu)$. We say the generalized mean curvature is perpendicular if $H(x)\bot T_x\Sigma$ for $\mu$-a.e. $x\in\Sigma$, where $T_x\Sigma$ is the approximating tangent space of the rectifiable set $\Sigma$.  For example, the generalized mean curvature of an integral varifold is perpendicular\cite{B-1978}

\begin{lemma}\label{Willmore lower bound}
Assume $V=\underline{v}(\Sigma,\theta)$ is a rectifiable $2$-varifold in $\mathbb{R}^n$ with perpendicular generalized mean curvature such that $\theta\ge 1$, $\mu(\mathbb{R}^n)<+\infty$ and $\int_{\Sigma}|H|^2d\mu<+\infty$. Then, by the monotonicity formula\cite{LS93, KS}, we know 
\begin{align}\label{dominate remainder term}
\Theta(\mu,x)=\lim_{r\to 0}\frac{\mu(B(x,r))}{\pi r^2} \text{ exists  for every } x\in \mathbb{R}^n,
\end{align}
and 
\begin{align*}
\Theta(\mu,x)= \frac{1}{16\pi}\int_{\mathbb{R}^n}|H|^2d\mu-\frac{1}{\pi}\int_{\mathbb{R}^n}|\frac{H}{4}+\frac{\nabla^{\bot}r}{r}|^2d\mu.
\end{align*}
Since $\Theta(\mu,x)=\theta(x)\ge 1$ for $\mu$-a.e. $x\in \Sigma$, there holds
\begin{align*}
\int_{\Sigma}|H|^2d\mu\ge 16\pi\Theta(\mu, x)\ge 16\pi.
\end{align*}
\end{lemma}
\begin{lemma}\label{compactness}
Assume $V=\underline{v}(\Sigma,\theta)$ is a rectifiable $2$-varifold in $\mathbb{R}^n$ such that $\theta\ge 1$, $H\in L^2(d\mu)$ and 
\begin{align*}
\mu(\mathbb{R}^n)<+\infty \text{ and } \int_{\mathbb{R}^n}|H|^2d\mu<32\pi.
\end{align*}
Then $\Sigma$ is a compact set with  
$$\frac{1}{7}\sqrt{\frac{\mu(\mathbb{R}^n)}{4\pi}}\le \text{diam}(\Sigma)<C\sqrt{\mu(\mathbb{R}^n)},$$
where $C$ is a universal constant. 
\end{lemma} 
\begin{proof} The proof is almost the same as  Leon Simon \cite{LS93} did in the smooth case, except for the issue of applying the condition $\int_{\Sigma}|H|^2d\mu<32\pi$ to guarantee some kind of connectivity of $\Sigma=\text{spt}\mu$ in the varifold case.

For any $x_i\in \Sigma$ such that $x_i\to x$, by $\theta\ge 1$ we know that $\Theta(\mu,x_i)\ge 1$. So, the upper semi-continuity from the monotonicity formula implies
$$\Theta(\mu,x)\ge \limsup_{i\to \infty}\Theta(\mu,x_i)\ge 1.$$
This means $x\in spt\mu=\Sigma$ and hence $\Sigma$ is a closed subset in $\mathbb{R}^n$.

Now, we prove $\Sigma$ is compact by showing  $\text{diam}(\Sigma)<+\infty$. Otherwise, we have $\text{diam}(\Sigma)=\infty$ and hence for fixed $x\in \Sigma$, $\rho>0$ and any $N\in \mathbb{N}_+$ and $1\le j\le N$, $\Sigma\backslash B(x,(j+\frac{1}{2})\rho)\neq\emptyset$.   We claim that there exists $y_j\in \Sigma\cap \partial B(x,(j+\frac{1}{2})\rho)$. Otherwise, letting $U=\breve{B}(x,(j+\frac{1}{2})\rho)$ and $V=\mathbb{R}^n\backslash B(x,(j+\frac{1}{2})\rho)$, then $U\cap V=\emptyset$ and $\Sigma\subset U\cup V$. Since $\Sigma$ is a closed set, we know 
\begin{align*}
\min\{d(U\cap \Sigma, \partial U), d(V\cap \Sigma, \partial V)\}>0.
\end{align*}
This means both $V_1=\underline{v}(\Sigma\cap U, \theta|_U)$ and $V_2=\underline{v}(\Sigma\cap V, \theta|_V)$ can be regarded as varifold in $\mathbb{R}^n$ with 
\begin{align*}
\mu_{V_i}(\mathbb{R}^n)+\int_{\mathbb{R}^n}|H|^2d\mu_{V_i}<+\infty, \quad \quad i=1,2.
\end{align*}
So, by Lemma \eqref{Willmore lower bound}, we know that 
\begin{align*}
\int_{\Sigma}|H|^2d\mu\ge \int_{\Sigma\cap U}|H|^2d\mu_{V_1}+\int_{\Sigma\cap V}|H|^2d\mu_{V_2}\ge 32 \pi, 
\end{align*}
which contradicts to the condition $\int_{\mathbb{R}^n}|H|^2d\mu< 32\pi$. We continue the proof of $\text{diam}(\Sigma)<+\infty$. By the monotonicity formula, we also have 
\begin{align*}
\Theta(\mu,y_j)\le C(\frac{\mu(B(y_j,\frac{\rho}{4}))}{\rho^2}+\int_{B(y_j,\frac{\rho}{4})}|H|^2d\mu).
\end{align*}
Since $y_j\in \partial B(x,(j+\frac{1}{2})\rho)$, we know that $B(y_j,\frac{\rho}{4})\cap B(y_k,\frac{\rho}{4})=\emptyset, \forall j\neq k$, and hence 
\begin{align}\label{upper bound of the diameter}
N\pi \le \sum_{j=1}^N\Theta(\mu,y_j)\le C(\frac{\mu(\mathbb{R}^n)}{\rho^2}+\int_{\mathbb{R}^n}|H|^2d\mu).
\end{align}
Letting $N\to \infty$, we get a contradiction. This means $\text{diam}(\Sigma)<+\infty$ and hence $\Sigma$ is compact.
So, we can assume $d:=\text{diam}(\Sigma)=|x-y|$ for some $x,y\in \Sigma$. Taking $\rho\in (0,\frac{d}{100}]$ and $N=[\frac{d}{\rho}]$, the same argument as above showing that  \eqref{upper bound of the diameter} still holds and hence 
\begin{align*}
\frac{d}{\rho}\le 2N\le C(\frac{\mu(\mathbb{R}^n)}{\rho^2}+\int_{\mathbb{R}^n}|H|^2d\mu).
\end{align*}
So, we get 
\begin{align*}
d\le  C(\frac{\mu(\mathbb{R}^n)}{\rho}+\rho\int_{\Sigma}|H|^2d\mu).
\end{align*}
Letting $\rho_0=\sqrt{\frac{\mu(\mathbb{R}^n)}{\int_{\Sigma}|H|^2d\mu}}$, we discuss by cases. If $\rho_0\le \frac{d}{100}$, then we know 
\begin{align*}
d\le C(\frac{\mu(\mathbb{R}^n)}{\rho_0}+\rho_0\int_{\Sigma}|H|^2d\mu)=C\sqrt{\int_{\Sigma}|H|^2d\mu \cdot \mu(\mathbb{R}^n)}\le C\sqrt{32\pi\cdot\mu(\mathbb{R}^n)}.
\end{align*}
If $\rho_0>\frac{d}{100}$, then we know 
\begin{align*}
d\le 100\rho_0=100\sqrt{\frac{\mu(\mathbb{R}^n)}{\int_{\Sigma}|H|^2d\mu}}\le 100\sqrt{\frac{\mu(\mathbb{R}^n)}{16\pi}}\le C\sqrt{\mu(\mathbb{R}^n)}.
\end{align*}
This complete the estimate of the upper bound of the diameter. For the lower bound, the monotonicity formula implies 
\begin{align*}
\frac{\mu(B_\sigma(x))}{\sigma^2}\le \frac{1}{16}\int_{\mathbb{R}^n}|H|^2d\mu+\frac{2}{\sigma}\int_{\mathbb{R}^n}|H|.
\end{align*}
By letting $\sigma=d=\text{diam}(\Sigma)$, we know 
\begin{align*}
\frac{\mu(\mathbb{R}^n)}{d^2}\le \frac{1}{16}\int_{\mathbb{R}^n}|H|^2d\mu+\frac{1}{\varepsilon}\int_{\mathbb{R}^n}|H|^2d\mu+\varepsilon\frac{\mu(\mathbb{R}^n)}{d^2},
\end{align*}
which implies (by taking $\varepsilon=\frac{1}{2}$)
\begin{align*}
\frac{\mu(\mathbb{R}^n)}{d^2}\le 5\int_{\mathbb{R}^{n}}|H|^2d\mu\le 160\pi.
\end{align*}
As a result, we know 
\begin{align*}
\frac{1}{7}\sqrt{\frac{\mu(\mathbb{R}^n)}{4\pi}}\le d.
\end{align*}
\end{proof}

 For a rectifiable $2$-varifold $V=\underline{v}(\Sigma,\theta)$  in $\mathbb{R}^n$ with perpendicular generalized mean curvature $H\in L^2(d\mu)$ and $\theta\ge 1$, Let $p\in \Sigma$ and $f:\mathbb{R}^n\backslash\{p\}\to \mathbb{R}^n\backslash \{p\}$ be the inversion map 
\begin{align*}
f(x):=f_p(x):=\frac{x-p}{|x-p|^2}+p.
\end{align*}
Noting $f$ is one-to one and proper, we know $\tilde{V}=f_{\sharp}V=\underline{v}(\tilde{\Sigma},\tilde{\theta})$ is a rectifiable  $2$-varifold in $\mathbb{R}^n$, where $\tilde{\Sigma}=f(\Sigma)$ and $\tilde{\theta}\circ f=\theta$, i.e., for $y=f(x)$,$\tilde{\theta}(y)=\theta(x)$. Moreover, by Lemma \eqref{compactness}, we know $f(\Sigma)$ is a closed set such that $p\notin f(\Sigma)$.

Next, we calculate the first variation and the mass density identity of the inverted varifold. The same formula in the smooth case can be found in \cite{Z22}. 
\begin{lemma}\label{inverted lemma}Assume $V=\underline{v}(\Sigma,\theta)$ is a rectifiable $2$-varifold in $\mathbb{R}^n$ with perpendicular generalized mean curvature $H\in L^2(d\mu)$, $\theta\ge 1$ and $\mu(\mathbb{R}^n)<+\infty$. Then, $\tilde{V}$ is a rectifiable varifold with perpendicular generalized mean curvature and  for any $\tilde{X}\in C_c^1(\mathbb{R}^n)$, there holds
\begin{align*}
\delta \tilde{V}(\tilde{X})=\int_{\tilde{\Sigma}}div^{\tilde{\Sigma}}\tilde{X} d\tilde{\mu}=-\int_{\tilde{\Sigma}}\tilde{H}\cdot \tilde{X}d\tilde{\mu},
\end{align*}
where $\tilde{\mu}=\tilde{\theta}(y)\mathcal{H}^2\llcorner \tilde{\Sigma}$, 
\begin{align}\label{inverted mean curvature}
\tilde{H}(f(x))=|x|^2R_x(H+4\frac{x^{\bot}}{|x|^2}),
\end{align}
and $R_x(v)=v-2\frac{1}{|x|^2}\langle v,x\rangle x$ is the reflection with respect to the direction $x$. Moreover, we have 
 \begin{align*}
 |\tilde{H}|^2d\tilde{\mu}(y)=|H+4\frac{x^\bot}{|x|^2}|^2d\mu(x),\quad 
 \text{  }\quad  \quad 
 \frac{|H|^2}{16}d\mu(x)=|\frac{\tilde{H}}{4}+\frac{y^{\bot}}{|y|^2}|^2d\tilde{\mu}(y)
 \end{align*}
  and 
  \begin{align*}
 \Theta(\tilde{\mu},\infty):=\lim_{r\to \infty}\frac{\tilde{\mu}(B(p,\rho))}{\pi \rho^2}=\Theta(\mu,p).
  \end{align*}
\end{lemma}
\begin{proof}
By Lemma \eqref{compactness}, we know $f(\Sigma)$ is a closed set properly embedded in $\mathbb{R}^m$ such that $p\notin f(\Sigma)$. Thus we only need to consider vector field $\tilde{X}\in C_c^1(\mathbb{R}^n\backslash \{p\})$. We first calculate $div^{\tilde{\Sigma}}\tilde{X}(y)=\text{tr}_{T_y\tilde{\Sigma}}D\tilde{X}(y)$. For this, assume $x\in \Sigma$ such that $y=f(x)$
 and $T_x\Sigma$ exists, letting $T_x\Sigma=\text{span}\{e_1,e_2\}$, then $T_y\tilde{\Sigma}=\text{span}\{f_*(e_i)\}_{i=1,2}$. Without loss of generality, assuming $p=0$, then $f(x)=\frac{x}{|x|^2}$ and we have 
 \begin{align}\label{differential of inversion}
 df(e_i)=\frac{R_x(e_i)}{|x|^2},
 \end{align}
 where $R_x(v)=v-2\langle \frac{x}{|x|},v\rangle \cdot \frac{x}{|x|}$ is the reflection with respect to the direction $\frac{x}{|x|}$. So,$T_y\tilde{\Sigma}=T_x\Sigma$ (with opposite orientation). 
 If we denote $h(y)=f^{-1}(y)=\frac{y}{|y|^2}$ and put $X(x)=|x|^2\tilde{X}(f(x))-2\langle \tilde{X}(f(x)),x\rangle x$, then 
 \begin{align*}
 X(x)=R_x(|x|^2\tilde{X}(f(x)))  \text{ and } \tilde{X}(y)=|y|^2R_y(X(h(y))). 
 \end{align*}
 Noting that $X\in C_c^1(\mathbb{R}^n\backslash\{0\})$, we know $X(x)=0$ for $x$ close to $0\in \mathbb{R}^n$. 
So, we have 
\begin{align*}
D\tilde{X}=2\langle y, dy\rangle \otimes X(h(y))&+|y|^2DX(h(y))\circ d_yh\\
&-2\langle DX\circ d_yh,y\rangle\otimes y-2\langle X(h(y)),dy\rangle\otimes y-2\langle X(h(y)),y\rangle dy
\end{align*}
 and 
 \begin{align*}
 div^{T_y\tilde{\Sigma}}\tilde{X}(y)&=\langle D_{e_i}\tilde{X},e_i\rangle\\
 &=2\langle y^{\top},X\rangle +\sum_{i}(\langle DX(R_y(e_i)),e_i\rangle -2\langle DX,R_y(e_i,\frac{y}{|y|})\rangle \langle \frac{y}{|y|},e_i\rangle) -2\langle X,y^{\top}\rangle-4\langle X,y\rangle\\
 &=\sum_{i=1,2}\langle R_y\circ DX\circ R_y(e_i),e_i\rangle-4\langle X,y\rangle\\
 &=\text{tr}_{T_x\Sigma}R_y\circ DX\circ R_y-4\langle X,y\rangle\\
 &=div^{T_x\Sigma}X-4\langle X,y\rangle.
 \end{align*}
 By the area formula for varifold and noting $J_f(x)=\frac{1}{|x|^4}$, we know that
 \begin{align*}
 \delta\tilde{V}(\tilde{X})&=\int_{\tilde{\Sigma}}div^{T_y\tilde{\Sigma}}\tilde{X}(y)d\tilde{\mu}(y)\\
 &=\int_{\Sigma}div^{T_x\Sigma}X-4\langle X,\frac{x}{|x|^2}\rangle \theta(x)J_f(x)d\mathcal{H}^2\llcorner \Sigma\\
 &=\int_{\Sigma}\frac{div^{T_x\Sigma}x}{|X|^{4}}-4\frac{\langle X,x\rangle}{|x|^6}d\mu(x)\\
 &=\int_{\Sigma}div^{T_x\Sigma}(\frac{X}{|x|^{4}})-4\frac{\langle X,x^{\bot}\rangle}{|x|^6}d\mu(x)\\
 &=-\int_{\Sigma}\frac{\langle X,H\rangle }{|x|^{4}}-4\frac{\langle X,x^{\bot}\rangle}{|x|^6}d\mu(x)\\
 &=-\int_{\tilde{\Sigma}}|x|^2\langle R_x(\tilde{X}),H\rangle-4\langle R_x(\tilde{X}),x^{\bot}\rangle d\tilde{\mu}(y)\\
 &=-\int_{\tilde{\Sigma}}\langle df_x(|x|^4(H+4\frac{x^{\bot}}{|x|^2})),\tilde{X}\rangle d\tilde{\mu}(y).
 \end{align*}
 This implies 
 $$\tilde{H}(y)=df_x(|x|^4(H+4\frac{x^{\bot}}{|x|^2}))=|x|^2R_x(H+4\frac{x^{\bot}}{|x|^2})$$
 and hence $\tilde{H}(y)\bot T_y\tilde{\Sigma}$ for $\tilde{\mu}$-a.e. $y\in \tilde{\Sigma}$ and 
 \begin{align*}
 |\tilde{H}|^2d\tilde{\mu}(y)=|H+4\frac{x^\bot}{|x|^2}|^2d\mu(x).
 \end{align*}
 For the same reason, we have 
 \begin{align*}
 \frac{|H|^2}{16}d\mu(x)=|\frac{\tilde{H}}{4}+\frac{y^{\bot}}{|y|^2}|^2d\tilde{\mu}(y).
 \end{align*}
 Noting that $p\notin \tilde{\Sigma}$, by the monotonicity identity again, we have 
 \begin{align*}
 0&=\frac{\tilde{\mu}(B(p,\rho))}{\rho^2}+\frac{1}{16}\int_{B(p,\rho)}|\tilde{H}|^2d\tilde{\mu}-\int_{B(p,\rho)}|\frac{\tilde{H}}{4}+\frac{\tilde{\nabla}^{\bot}\tilde{r}}{\tilde{r}}|^2d\tilde{\mu}+\frac{2}{\rho^2}\int_{B(p,\rho)}\langle \tilde{r}\tilde{\nabla}^{\bot}\tilde{r},\tilde{H}\rangle d\tilde{\mu}\\
 &\ge \frac{\tilde{\mu(B(p,\rho))}}{2\rho^2}-(2-\frac{1}{16})\cdot 16\int_{B(p,\rho)}|\frac{H}{4}+\frac{\nabla^{\bot}r}{r}|^2d\mu-\frac{1}{16}\int_{\mathbb{R}^n}|H|^2d\mu.
 \end{align*}
 Substituting \eqref{dominate remainder term} into this inequality, we get
 \begin{align*}
\pi \Theta^*(\tilde{\mu},+\infty):=\limsup_{\rho\to \infty}\frac{\tilde{\mu}(B(p,\rho))}{\rho^2}\le 5\int_{\mathbb{R}^n}|H|^2d\mu<+\infty.
 \end{align*}
 So, by letting $\rho\to \infty$ in 
 \begin{align*}
 |\frac{2}{\rho
^2}\int_{B_\rho}\langle \tilde{r}\tilde{\nabla}^{\bot}\tilde{r},\tilde{H}\rangle|^2d\tilde{\mu}|&\le \frac{1}{\rho^2}(\int_{B_R}\rho|\tilde{H}|d\tilde{\mu}+\int_{B_\rho\backslash B_R}|\tilde{H}|d\tilde{\mu})\\
&\le \left(\frac{\tilde{\mu}(B_R)}{\rho^2}\right)^{\frac{1}{2}} \left(\int_{B_R}|\tilde{H}|^2d\tilde{\mu}\right)^{\frac{1}{2}}+ \left(\frac{\tilde{\mu}\left(B_\rho\backslash B_R\right)}
{\rho^2}\right)^{\frac{1}{2}} \left(\int_{B_\rho\backslash B_R}|\tilde{H}|^2d\tilde{\mu}\right)^{\frac{1}{2}},
 \end{align*}
  we get 
  \begin{align*}
  \limsup_{\rho\to \infty} |\frac{2}{\rho
^2}\int_{B_\rho}\langle \tilde{r}\tilde{\nabla}^{\bot}\tilde{r},\tilde{H}\rangle|^2d\tilde{\mu}|\le\Theta^*(\tilde{\mu},\infty)^{\frac{1}{2}}\left(\int_{\mathbb{R}^n\backslash B_R}|\tilde{H}|^2d\tilde{\mu}\right)^{\frac{1}{2}}.
  \end{align*}
  Letting $R\to \infty$, we know 
  \begin{align*}
  \limsup_{\rho\to \infty} |\frac{2}{\rho
^2}\int_{B_\rho}\langle \tilde{r}\tilde{\nabla}^{\bot}\tilde{r},\tilde{H}\rangle|^2d\tilde{\mu}|=0,
  \end{align*}
   and hence by combining  the monotonicity formulae  of $\tilde{V}$ and $V$ together, we know 
   \begin{align*}
  \lim_{\rho\to \infty} \frac{\tilde{\mu}(B(p,\rho))}{\rho^2}&=\lim_{\rho\to \infty}\left[-\frac{1}{16}\int_{B(p,\rho)}|\tilde{H}|^2d\tilde{\mu}+\int_{B(p,\rho)}|\frac{\tilde{H}}{4}+\frac{\tilde{\nabla}^{\bot}\tilde{r}}{\tilde{r}}|^2d\tilde{\mu}\right]\\
  &=\lim_{\rho\to \infty}\left[\frac{1}{16}\int_{\mathbb{R}^n\backslash B(p,\frac{1}{\rho})}|H|^2d\mu-\int_{\mathbb{R}^n\backslash B(p,\frac{1}{\rho})}|\frac{H}{4}+\frac{\nabla^{\bot}r}{r}|^2d\mu\right]\\
  &=\pi\Theta(\mu,p).
  \end{align*}
  Thus the limit exists and $\Theta(\tilde{\mu},\infty)=\Theta(\mu,p)$.
\end{proof}

As a corollary, we have
\begin{corollary}\label{Critical Allard condition for the inverted varifold}Assume $V=\underline{v}(\Sigma,\theta)$ is a rectifiable $2$-varifold in $\mathbb{R}^n$ with perpendicular generalized mean curvature $H\in L^2(d\mu)$, $\theta\ge 1$ and $\mu(\mathbb{R}^m)<+\infty$.  If 
$$\int_{\mathbb{R}^n}|H|^2d\mu\le 16\pi(1+\delta^2)$$
for some $\delta<1$, then for any $p\in \Sigma$, the inverted varifold $\tilde{V}=f_{p\sharp}V=\underline{v}(\tilde{\Sigma},\tilde{\theta})$ satisfies
\begin{align*}
\int_{\mathbb{R}^n}|\tilde{H}|^2d\tilde{\mu}\le 16\pi\delta^2\quad \quad  \text{ and } \quad \quad 
\Theta(\tilde{\mu},\infty)\le 1+\delta^2.
\end{align*}
In particular, $\int_{\Sigma}|H|^2d\mu\leq 16\pi$ holds if and only if $\theta\equiv 1$ and $\Sigma$ is a round sphere.
\end{corollary}
\begin{proof} Since $\theta\ge 1$ for $\mu$-a.e. $x\in \mathbb{R}^n$,  we know for any $p\in \Sigma$, there holds $\Theta(\mu,p)\ge 1$. Denote $f_p(x)=\frac{x-p}{|x-p|^2}+p$ and the inverted varifold by $f_{p\sharp}V=\tilde{V}=\underline{v}(\tilde{\Sigma},\tilde{\theta})$. Then, by the monotonicity formula, we know 
\begin{align*}
1\le \Theta(\mu,p)=\frac{1}{16\pi}\int_{\mathbb{R}^n}|H|^2d\mu-\frac{1}{\pi}\int_{\mathbb{R}^n}\big|\frac{H}{4}+\frac{\nabla^{\bot}r}{r}\big|^2d\mu,
\end{align*}
where $r(x)=|x-p|$. Thus we have 
\begin{align*}
\int_{\mathbb{R}^n}|H|^2d\mu\ge 16\pi.
\end{align*}
If $\int_{\mathbb{R}^n}|H|^2d\mu\le 16\pi(1+\delta^2)$,  then Lemma \ref{inverted lemma} implies 
 \begin{align*}
 \int_{\mathbb{R}^n}|\tilde{H}|^2d\tilde{\mu}=\int_{\mathbb{R}^n}|\frac{H}{4}+\frac{x^\bot}{|x|^2}|^2d\mu=\frac{1}{16}\int_{\mathbb{R}^n}|H|^2d\mu-\Theta(\mu,p)\le \delta^2.
 \end{align*}
 and 
 \begin{align*}
 \Theta(\tilde{\mu},\infty)=\Theta(\mu,p)\le \frac{1}{16\pi}\int_{\mathbb{R}^n}|H|^2d\mu\le 1+\delta^2.
 \end{align*}
 In particular, when $\int_{\mathbb{R}^m}|H|^2d\mu=16\pi$, we know $\delta=0$ and hence 
 $$ \int_{\mathbb{R}^n}|\tilde{H}|^2d\tilde{\mu}=0 \text{ and } \Theta(\tilde{\mu},\infty)=1,$$
 i.e., $\tilde{V}$ is a stationary varifold with $\Theta(\tilde{\mu},\infty)=1$. Now Allard's regularity theorem implies $\tilde{\Sigma}$ is a plane with $\tilde{\theta}\equiv 1$. This means $\theta\equiv 1$ and $\Sigma$ is a round sphere.
\end{proof}

\section{Quantitative rigidity of the inverted varifold}\label{sec:Quantitative plane}
Next we analysis the quantitative rigidity for the inverted varifold. From now on, we need to assume the varifold to be integral. 

\begin{theorem}\label{inverted quantitative rigidity}There exists $\delta_0>0$ such that for any $\delta<\delta_0$ and any integral varifold  $\tilde{V}=\underline{v}(\tilde{\Sigma},\tilde{\theta})$  in $\mathbb{R}^n$ with
\begin{align*}
\int_{\mathbb{R}^n}|\tilde{H}|^2d\tilde{\mu}\le 16\pi\delta^2\quad \quad  \text{ and } \quad \quad 
\Theta(\tilde{\mu},\infty)\le 1+\delta^2,
\end{align*}
 there exists a $W^{2,2}_{loc}$ conformal parameterization $f:\mathbb{R}^2\to \tilde{\Sigma}$ such that $df\otimes df=e^{2w}g_{euc}$, $f(0)=0$,
\begin{align}\label{quantitative C^0}
\|w\|_{C^0(\mathbb{R}^2)}+\|Dw\|_{L^2(\mathbb{R}^2)}+\|D^2w\|_{L^1(\mathbb{R}^2)}+\|D^2f\|^2_{L^2(\mathbb{R}^2)}+\int_{\tilde{\Sigma}}|\tilde{A}|^2d\tilde{\mu}\le C \delta^2,
\end{align}
and
\begin{align}\label{quantitative bilipschitz}
(1-C\delta)|x-y|\le |f(x)-f(y)|\le (1+C\delta^2)|x-y|,
\end{align}
where $\tilde{A}$ is the second fundamental form of $\tilde{\Sigma}$ which is well-defined as the result and $C=C(n)$ is a constant.   Moreover,for any $r>0$ and $y\in \mathbb{R}^2$,  there exists $P_{r,y}\in Hom(\mathbb{R}^2,\mathbb{R}^n)$ such that $P_{r,y}^{*}P_{r,y}=id_{\mathbb{R}^2}$ and 
\begin{align}\label{approximate}
|f(x)-f(y)-P_{r,y}(x-y)|\le C\delta(r+\frac{|x-y|^2}{r}), \quad \quad \forall x\in \mathbb{R}^2.
\end{align}
\end{theorem}
\begin{proof}
Step 1.
Since $\tilde{V}$ is an integral varifold with  $\Theta(\tilde{\mu},\infty)\le 1+\delta^2$ and $\int_{\mathbb{R}^n}|\tilde{H}|^2d\tilde{\mu}\le 16\pi\delta^2$, by Theorem \ref{2022b-main}, we know that for fixed $p\in\tilde{\Sigma}$ and  $R\gg 1$,  there exist a bi-Lipschitz conformal parameterization $f_R: D(0,R)\to \tilde{\Sigma}$ such that $df_R\otimes df_R=e^{2w_R}(dx^1\otimes dx^1+dx^2\otimes dx^2)$ satisfying
\begin{enumerate}
\item $\tilde{\Sigma}\cap B(0,(1-\psi)R)\subset f_R(D(0,R))\subset B(0,(1+\psi)R)\cap \tilde{\Sigma}$, $\quad f_R(0)=p$,\\
 \item $\|w_R\|_{L^\infty(D_R)}\le \psi, \quad \int_{D(0,R)}|D^2f_R|^2\le \psi$,\\
 \item $(1-\psi)|x-y|\le |f_R(x)-f_R(y)|\le (1+\psi)|x-y|,\forall x,y\in D(0,R).$
\end{enumerate}
By Arzel\`{a}-Ascoli lemma, letting $R\to \infty$, we know there exists a subsequence $f_{R_i}$ converges weakly in $W^{2,2}$ to a limit $f:\mathbb{C}\to \tilde{\Sigma}$ such that $f\in W^{2,2}_{conf, loc}(\mathbb{C},\mathbb{R}^n)$ is a $W^{2,2}_{loc}$-conformal immersion(See \cite{KL-2012} for definition) satisfying 
\begin{enumerate}
\item $\tilde{\Sigma}=f(\mathbb{C})$,\\
\item $df\otimes df=e^{2w}(dx^1\otimes dx^1+dx^2\otimes dx^2)$ with $\|w\|_{L^{\infty}}\le \psi$\\
\item $(1-\psi)|x-y|\le |f(x)-f(y)|\le(1+\psi) |x-y|$,\\
\item $\int_{\mathbb{C}}|D^2f|^2dx^1dx^2\le \psi.$
\end{enumerate}
Moreover, noting that \begin{align*}
|D^2f|^2-|\Delta f|^2=\text{div}X,
\end{align*}
for $X=f_{ij}f_i\partial_j-f_{jj}f_i\partial_i$, we know  that for any cut-off function $\eta\in W_0^{1,2}(\mathbb{C})$, there holds 
\begin{align*}
|\int_{\mathbb{C}}(D^2f-(\Delta f)^2)\eta|=|\int_{\mathbb{C}}-f_{ij}f_i\eta_j+f_{jj}f_i\eta_i|\le Lip f\left(\int_{\mathbb{C}}|D^2f|^2\int_{\mathbb{C}}|D\eta|^2\right)^{\frac{1}{2}}\le C\|D\eta\|_{L^2(\mathbb{C})}.
\end{align*}
By taking logarithm cut-off function and applying dominating convergence theorem, we know 
\begin{align}\label{Hessian estimate}
\int_{\mathbb{C}}|D^2f|^2=\int_{\mathbb{C}}|\Delta f|^2\le \sup e^{2w}\int_{\mathbb{C}}|H|^2e^{2w}dx^1dx^2\le 16\pi  e^{2\psi}\delta^2.
\end{align}
This further implies  
\begin{align*}
\int_{\tilde{\Sigma}}|\tilde{A}|^2d\tilde{\mu}\le \int_{\mathbb{C}}|D^2f|^2e^{-2w}\le e^{2\psi}\int_{\mathbb{C}}|D^2f|^2dx^1dx^2\le 16\pi e^{4\psi}\delta^2.
\end{align*}
Next, we are going to estimate the $L^\infty$ norm of the conformal factor $e^{2w}$ in a quantitative way. For this, the well-known approach is to apply M\"uller and \v{S}ver\'{a}k\cite{MS95}'s observation on the $\mathcal{H}^1$(Hardy norm) estimate of the Gauss curvature for immersed surfaces with small total curvature. In our case, since we have already obtained the $W^{2,2}$ estimate of $f$, there is an alternative approach using the Bochner formula for conformal immersions. 
\begin{pro}\label{New Hardy structure}
Assume $f\in W^{2,2}_{conf, loc}(\mathbb{C},\mathbb{R}^n)$ is a conformal immersion, i.e.,
$$\langle f_i, f_j\rangle =e^{2w}\delta_{ij},$$
then for  $v=e^{2w}$,  we have
\begin{align*}
\Delta v=-2\sum_{\alpha=1}^ndet(D^2f^\alpha).
\end{align*}
Moreover, if $\|D^2f\|_{L^2(\mathbb{C})}<+\infty$, then  there is a non-negative constant $v(\infty)$ such that $v_0=v-v(\infty)\in W_0^{1,2}\cap W^{2,1}\cap C^{0}$ satisfies
\begin{align}\label{zero boundary}
\lim_{z\to \infty}v_0(z)=0
\end{align}
and
\begin{align}\label{esimate}
\|v_0\|_{C^0(\mathbb{C})}+\|Dv_0\|_{L^2(\mathbb{C})}+\|D^2v_0\|_{L^1(\mathbb{C})}\le C \|D^2f\|^2_{L^2(\mathbb{C})}.
\end{align}
\end{pro}
\begin{proof}
Since $v=\frac{1}{2}|Df|^2$, by Bochner's formula, we have
\begin{align*}
\Delta v&=\sum_{\alpha=1}^{n}\frac{1}{2}\Delta|D f^{\alpha}|^2=\sum_{\alpha=1}^{n}|Hess f^\alpha|^2-\sum_{\alpha=1}^{n}\sum_{i=1}^{2}D_i\Delta f^\alpha D_if^{\alpha}.
\end{align*}
Note that
$$\sum_{\alpha=1}^{n}\sum_{i=1}^{2}D_i\Delta f^\alpha D_if^{\alpha}=\sum_{i=1}^2D_i\langle \Delta f,f_i\rangle-\sum_{\alpha=1}^n\Delta f^{\alpha}\sum_{i=1}^2D_{ii}f^\alpha=-\sum_{\alpha=1}^n(\Delta f^\alpha)^2,$$
where in the last equality we use the mean curvature equation under the conformal coordinate which implies
$$\langle \Delta f, f_i\rangle =\langle \vec{H}e^{2u}, f_i\rangle =0.$$
So, we get
\begin{align*}
\Delta v=\sum_{\alpha=1}^{n}|Hess f^\alpha|^2-(\Delta f^\alpha)^2=-2\sum_{\alpha=1}^ndet(D^2f^\alpha).
\end{align*}
Since $D^2f^\alpha\in L^2$,by \cite{CLMS-1993} we know $det(D^2f^\alpha)\in \mathcal{H}^1(\mathbb{C})$--the Hardy space and
\begin{align*}
\|det(D^2f^{\alpha})\|_{\mathcal{H}^1(\mathbb{C})}\le C\|D^2f\|_{L^2(\mathbb{C})}^2,
\end{align*}
and hence we can solve $\Delta v_0= -2\sum_{\alpha=1}^ndet(D^2f^\alpha)$ with boundary condition (\ref{zero boundary}) such that $v_0$ has the estimate
\begin{align*}
\|v_0\|_{C^0(\mathbb{C})}+\|Dv_0\|_{L^2(\mathbb{C})}+\|D^2v_0\|_{L^1(\mathbb{C})}\le C \|det(D^2f^{\alpha})\|_{\mathcal{H}^1(\mathbb{C})}.
\end{align*}
So, $h=v-v_0$ is a harmonic function. Moreover, since $v\ge 0$ and $\lim_{z\to \infty}v_0(z)=0$, we know that $h$ has lower bound. So, by Liouville's theorem, we know $h$ is a constant, which implies $v(\infty)=\lim_{z\to \infty}v(z)$ exists and the conclusion holds.
\end{proof}

Now, we continue to prove Theorem \ref{inverted quantitative rigidity}.  By Proposition \ref{New Hardy structure} and the estimate $\|w\|_{L^\infty}\le \psi$, we know 
$e^{2w}=v_0+v(\infty)$
such that \eqref{zero boundary} and \eqref{esimate} hold and 
\begin{align*}
v(\infty)=\lim_{x\to \infty}e^{2w}\in (1-\psi,1+\psi).
\end{align*}
So, without loss of generality (replace $f$ by $\tilde{f}(x)=f(\frac{x}{\sqrt{v(\infty)}})$), we can assume $$v(\infty)=\lim_{x\to \infty}\frac{1}{2}|df|^2(x)=1$$ and the four items $(1),(2),(3),(4)$ for $f$ still hold. Thus by \eqref{Hessian estimate},\eqref{zero boundary} and \eqref{esimate}, we know 
$e^{2w}=1+v_0$ with $v_0(\infty)=0$ and
\begin{align*}
\|v_0\|_{C^0(\mathbb{C})}+\|Dv_0\|_{L^2(\mathbb{C})}+\|D^2v_0\|_{L^1(\mathbb{C})}\le C\cdot 16\pi e^{2\pi}\delta^2.
\end{align*}
This implies $w=\frac{1}{2}\log{(1+v_0)}$ satisfies $w(\infty)=0$ and 
\begin{align*}
\|w\|_{C^0(\mathbb{C})}+\|Dw\|_{L^2(\mathbb{C})}+\|D^2w\|_{L^1(\mathbb{C})}\le C\cdot 16\pi e^{2\pi}\delta^2.
\end{align*}
As a result, we know 
\begin{align*}
|f(x)-f(y)|\le d_g(f(x),f(y))\le \int_x^ye^{w}\le (1+C \delta^2)|x-y|.
\end{align*}
For the lower bound, we first prove (\ref{approximate}), i.e.,   $f$ is well-approximated by orthogonal immersion.
To do this, we consider $f_{r,y}(z)=\frac{1}{r}[f(rz)-f(y)]$, where $r>0$ and $z=\frac{x-y}{r}$. Then, we calculate 
\begin{align*}
f_{r,y}(0)=0,  \quad \quad Df_{r,y}(z)=Df(rz+y) \quad \text{ and } \quad  D^2f_{r,y}(z)=rD^2f(rz+y),
\end{align*}
which implies 
\begin{align*}
f_{r,y}^*g_{\mathbb{R}^n}=e^{2w(rz+y)}g_{\mathbb{R}^2} \quad \text{ and } \|D^2f_{r,y}\|_{L^2(\mathbb{R}^2)}=\|D^2f\|_{L^2(\mathbb{R}^2)}\le C\delta.
\end{align*}
 
For any $R>0$, define $A_R:\mathbb{R}^2\to \mathbb{R}^n$ by $A_R(z):= \fint_{D_R}\nabla_zf_{r,y}(w)dw$, $g_R(z)=f_{r,y}(z)-A_R(z)$ and $h_R:=\nabla g_R $. Then $\fint_{D_R}h_R=0$ and $\nabla h_R=D^2f_{r,y}$ satisfies 
\begin{align*}
\int_{\mathbb{R}^2}|\nabla h_R|^2(z)dz\le C\delta^2.
\end{align*}
 Hence for $\forall p>2$, the Sobolev inequality implies 
\begin{align}\label{W1p estimate}
\big(\frac{1}{R^2}\int_{D_R}|h_R|^p(z)dz\big)^{\frac{1}{p}}\le C\delta.
\end{align}
By the Sobolev embedding $W^{1,p}\subset C^{\alpha=1-\frac{2}{p}}$, we know for any $w,z\in D_R$, there holds 
\begin{align*}
\frac{|g_R(w)-g_R(z)|}{|w-z|^\alpha}\le C\delta (\int_{D_R}|\nabla g_R|^p(z)dz)^{\frac{1}{p}}\le C\delta R^{1-\alpha}.
\end{align*}
Since $g_R(0)=0$, we know 
\begin{align*}
|f_{r,y}(z)-A_Rz|\le C\delta R^{1-\alpha}|z|^{\alpha}, \quad\forall z\in D_R.
\end{align*}
By \ref{W1p estimate} and Chebyshev's inequality, we know that $\Omega_R:=\{z\in D_R: \|\nabla f(z)-A_R\|\le 100C\delta\}$ has positive measure. Noting $\langle f_{r,y,i}(z),f_{r,y,j}(z)\rangle= e^{2w}(r z+y)\delta_{ij}$ for a.e. $z\in \mathbb{R}^2$, we can choose $z_0\in \Omega$ such that 
$\|A_R-\nabla f_{r,y}(z_0)\|\le 100 C\delta$ and $\nabla f_{r,y}(x_0)=(\eta_1,\eta_2)$ satisfies $\langle \eta_i,\eta_j\rangle =e^{2w(rz_0+y)}\delta_{ij}$.  Letting $P_R:\mathbb{R}^2\to \mathbb{R}^n$ be the linear map defined by $P_R(z)=e^{-w(r z_0+y)}(z^1\eta_1+z^2\eta_2)$, then $P_R$ is an orthogonal immersion and 
\begin{align*}\|A_R-P_R\|&\le \|A_R-\nabla f_{r,y}(z_0)\|+\|\nabla f_{r,y}(z_0)-e^{-w(rz_0+y)}\nabla f_{r,y}(z_0)\|\\
&\le C\delta+\sqrt{2}\|e^w\|_{C^{0}}\|e^{-w}-1\|_{C^0}\\
&\le C\delta.
\end{align*}
So, we know 
\begin{align}\label{orthogonal approximation}
|f_{r,y}(z)-P_R(z)|\le C\delta R^{1-\alpha}|z|^\alpha+C\delta |z|\le C\delta R^{1-\alpha}|z|^\alpha, \forall z\in D_R.
\end{align}
Now, can show that \eqref{approximate} holds for $P=P_1$.

More precisely, for $|z|\le 1$, it follows  directly from \eqref{orthogonal approximation} that 
$$|f_{r,y}(z)-P_R(z)|\le C\delta\le C\delta(1+|z|^2).$$
For $|z|>1$, letting $R_i=2^i$ and we know there exists $i_0\ge 0$ such that $R_{i_0}<|z|\le R_{i_0+1}$. By \eqref{orthogonal approximation}, we know that 
\begin{align*}
|P_{R_i}(z)-P_{R_{i+1}}(z)|\le 3C\delta R_i, \forall |z|\le R_i,
\end{align*}
which implies $\|P_{R_i}-P_{R_i}\|\le 3C\delta$ and hence
$$\|P_1-P_{i_0+1}\|\le\sum_{i=1}^{i_0}\|P_{R_i}-P_{R_{i+1}}\|\le C \log_{2}(|z|)\delta. $$ As a result, we get 
\begin{align*}
|f_{r,y}(z)-P_1(z)|&\le |f_{r,y}(z)-P_{R_{i_0+1}}(z)|+\|P_1-P_{R_{i_0+1}}\||z|\\
&\le C\delta R_{i_0+1}+C\delta \log_{2}(|z|)|z|\\
&\le C\delta(1+|z|^2).
\end{align*}
By the definition of $f_{r,y}$, we get 
\begin{align*}
|f(x)-f(y)-P(x-y)|\le C\delta(r+\frac{|x-y|^2}{r}).
\end{align*}
Finally, by choosing $r=|x-y|$, we get 
\begin{align*}
|f(x)-f(y)|\ge |P(x-y)|-2C\delta |x-y|\ge (1-C\delta)|x-y|,\quad \forall x\in \mathbb{R}^2.
\end{align*}
This finishes the proof of Theorem \ref{inverted quantitative rigidity}.
\end{proof}
\begin{remark}\label{vanishing of Gauss curvature} In the above quantitative bi-Lipschitz estimate \eqref{quantitative bilipschitz}, the upper bound $1+C\delta^2$  comes from the intrinsic estimate 
$\|w\|_{C^0}\le C\delta^2$ while the lower bound $1-C\delta$  comes from the extrinsic estimate $\|D^2f\|_{L^2}\le C\delta$. Though this lower bound is enough for our later application, we remark it can be improved to be  $1-C\delta^2$.  In fact, by
 the proof of  \cite[Theorem 4.3.1]{MS-1995}, the upper bound and lower bound can be interchanged through an inversion $\tilde{f}(z)=|f(\frac{1}{\bar{z}})|^{-2}f(\frac{1}{\bar{z}})$ and hence we  know 
$$|f(x)-f(y)|\ge (1-C\delta^2)|x-y|.$$
To apply \cite[Theorem 4.3.1]{MS-1995}(see also \cite[Theorem 5.1]{Sch13}), we check the assumption $\int_{\tilde{\Sigma}}Kd\tilde{\mu}=0$ as following.    By \cite[Remark 2.1]{KL-2012}, the Gauss curvature $K=\frac{1}{2}(|\tilde{H}|^2-|\tilde{A}|^2)$ of a $W^{2,2}_{loc}$ conformal immersion satisfies the Liouville equation 
\begin{align*}
-\Delta w=Ke^{2w}\in L^1. 
\end{align*}
Noting $|Dw|^2\le |D^2f|^2e^{-2w}$ and $\|w\|_{C^0}\le \psi$, we know $\int_{\mathbb{C}}|Dw|^2\le C \int_{\mathbb{C}}|D^2f|^2\le \psi$.
So, by taking logarithm cut-off function $\eta_R$ and applying dominating convergence theorem, we know 
\begin{align*}
|\int_{\mathbb{C}}Ke^{2w}|=|\lim_{R\to \infty}\int_{\mathbb{C}}\nabla w\nabla \eta_R|\le (\psi\lim_{R\to \infty}\int_{\mathbb{C}}|\nabla \eta_R|^2)^{\frac{1}{2}}=0.
\end{align*}
This also implies $\int_{\tilde{\Sigma}}|\tilde{A}|^2d\tilde{\mu}=\int_{\tilde{\Sigma}}|\tilde{H}|^2d\tilde{\mu}\le 16\pi\delta^2$.
\end{remark}

\section{Proof of Theorem \ref{main theorem}}\label{sec:Quantitative sphere}
Now, we are prepared to prove Theorem \ref{main theorem}. Actually,  by Theorem \ref{inverted quantitative rigidity}, we know the varifold $V$ in Theorem \ref{main theorem} is  a topological sphere with $W^{2,2}$ conformal paremetrization. Hence, one can adapt the argument of \cite{LS-2014}  to prove Theorem \ref{main theorem}. However, since we have obtained the sharp quantitative estimates\eqref{quantitative C^0} \eqref{quantitative bilipschitz} for the inverted varifold, it is more direct to invert this inverted varifold back, and show the quantitative rigidity of $V$.  

\begin{theorem}\label{partial Quantitative ridigity}
Assume $V=\underline{v}(\Sigma,\theta)$ is an integral varifold in $\mathbb{R}^n$  with generalized mean curvature $H\in L^2(d\mu_V)$, $\theta\ge 1$ and satisfying 
\begin{align*}
\mu_V(\mathbb{R}^n)=4\pi\quad \quad \text{ and } \int_{\mathbb{R^n}}|H|^2d\mu_V\le 16\pi(1+\delta^2).
\end{align*}
Then, for $\delta<\delta_0\ll 1$, there exists a round sphere $\mathbb{S}^2$ with radius equal to one and centered at $0\in \mathbb{R}^n$ such that after translation, there exists a conformal parameterization $\Phi: \mathbb{S}^2\to \Sigma$ such that $d\Phi\otimes d\Phi=e^{2v}g_{\mathbb{S}^2}$ and 
\begin{align}\label{C0 estimate}
\sup_{q\in \mathbb{S}^2}|\Phi(q)-q|+\|v\|_{C^0(\mathbb{S}^2)}\le C\delta 
\end{align}
\end{theorem}
\begin{proof}
By Lemma \ref{compactness}, we know
\begin{align}\label{diameter estimate}
\frac{1}{7}\le d:=\text{diam}(\Sigma)\le C(n)<\infty.
\end{align}
By translation, we can assume $0\in \textup{spt }\mu_V$ and there exists $p_0\in\Sigma$ such that $d(0,p_0)=d$.  Then, we know $d(p,0)\le d$ holds for any $p\in \Sigma$. Consider the inversion $\pi: \mathbb{R}^n\to \mathbb{R}^n$ defined by $\pi(p)=\frac{p}{|p|^2}$ and the inverted varifold $\tilde{V}=\pi_{\sharp}V=\underline{v}(\tilde{\Sigma},\tilde{\theta})$. Then, Corollary \ref{Critical Allard condition for the inverted varifold} implies 
\begin{align*}
\int_{\mathbb{R}^n}|\tilde{H}|^2d\tilde{\mu}\le 16\pi\delta^2\quad \quad  \text{ and } \quad \quad 
\Theta(\tilde{\mu},\infty)\le 1+\delta^2.
\end{align*}
Denote $v=\pi(p_0)=\frac{p_0}{|p_0|^2}$. Then, for any $\tilde{p}\in \tilde{\Sigma}$, we have  

\begin{align}\label{distance lower bound}
|\tilde{p}|\ge |v|=\frac{1}{d}\ge \frac{1}{C}.
\end{align}
Define the translation  $\tau_v(\cdot )=\cdot -v$ and denote $\hat{V}=\tau_{v\sharp}\tilde{V}=\underline{v}(\hat{\Sigma},\hat{\theta})$. Then, $\hat{V}$ is a varifold passing through the origin with the Critical Allard condition holds. Thus by Theorem \ref{inverted quantitative rigidity}, we know there exists a conformal immersion $f\in W^{2,2}_{loc}(\mathbb{R}^2, \hat{\Sigma})$ such that $f(0)=0$, $df\otimes df=e^{2w(x)}(dx^1\otimes dx^1+dx^2\otimes dx^2)$ with 
\begin{align}\label{intrinsic estimate}
\|w\|_{C^0(\mathbb{R}^2)}+\|Dw\|_{L^2(\mathbb{R}^2)}+\|D^2w\|_{L^1(\mathbb{R}^2)}+\|D^2f\|^2_{L^2(\mathbb{R}^2)}+\int_{\hat{\Sigma}}|\hat{A}|^2d\tilde{\mu}\le C \delta^2,
\end{align}
and
\begin{align}\label{bi-Lipschitz}
(1-C\delta)|x-y|\le |f(x)-f(y)|\le (1+C\delta^2)|x-y|,\forall x,y\in \mathbb{R}^2.
\end{align}
Moreover, there exists $P\in Hom(\mathbb{R}^2,\mathbb{R}^n)$ such that $P^{*}P=id_{\mathbb{R}^2}$ and 
\begin{align*}
|f(x)-P(x)|\le C\delta(1+|x|^2), \quad \quad \forall x\in \mathbb{R}^2.
\end{align*}

This means the support $\tilde{\Sigma}$ of the inverted varifold is close to the plane $P(\mathbb{R}^2)+v$. To show how far is the plane away from the inverted point $0\in \mathbb{R}^n$, we consider the point $x_0\in \mathbb{R}^2$  such that  
$$|P(x_0)+v|=\min_{x\in \mathbb{R}^2} |P(x)+v|,$$
denote $v'=P(x_0)+v$. Then we know $d(0,v')=d(0, P(\mathbb{R}^2)+v)$ and hence 
 \begin{align}\label{perpendicular}
 v'\bot P(\mathbb{R}^2).
 \end{align}
We will use the inversion of  this plane to construct a sphere and a map from this sphere to $\mathbb{R}^n$ which approximate the support $\Sigma$ of $V$. To do this, we need to show $P(\mathbb{R}^2)+v$ does not pass through the origin.   More precisely,  by \eqref{distance lower bound} and noting $|P(x)|=|x|$, we know 

\begin{align*}
|P(x)+v|\ge |x|-|v|\ge |x|-\frac{1}{d}.
\end{align*}
So for $|x|\ge \frac{2}{d}$, we know $|P(x)+v|\ge \frac{1}{d}$. For $|x|<\frac{2}{d}$,   (\ref{diameter estimate}) implies $|x|\le 14$. Thus,  \eqref{approximate} and \eqref{distance lower bound} implies
\begin{align*}
|P(x)+v|\ge |f(x)+v|-C\delta(1+|x|^2)\ge \frac{1}{d}-C\delta(1+|x|^2)\ge \frac{1}{d}-200C\delta,
\end{align*}
 and hence for $\delta\le \delta_0=\delta_0(C)\ll 1$, we know $|P(x)+v|\ge \frac{1}{2d}, \forall |x|\le \frac{2}{d}$. In conclusion, we know 
\begin{align}\label{away from zero}
|P(x)+v|\ge \frac{1}{2d}\ge \frac{1}{C}, \forall x\in \mathbb{R}^2.
\end{align}
 This means $P(\mathbb{R}^2)+v$ is away from $0$ and hence $S_r=\pi(P(\mathbb{R}^2)+v)$ is a sphere passing through the origin with diameter $2r=\text{diam}(S_r)=|\pi(v')|=\frac{1}{|v'|}\le 2d$.  

 Next, we will construct a conformal parameterization from $S_r$ to $\tilde{\Sigma}$.  Before this, we denote  $A(x)=P(x+x_0)+v$ and $\tilde{f}(x)=f(x+x_0)+v$. Then, $\tilde{f}(\mathbb{R}^2)=f(\mathbb{R}^2)+v= \tilde{\Sigma}$, i.e., $\tilde{f}$ is also a parameterization of $\tilde{\Sigma}$. By \eqref{bi-Lipschitz}, we know  
\begin{align}\label{tilde bi-Lipschitz}
(1-C\delta)|x-y|\le |\tilde{f}(x)-\tilde{f}(y)|\le (1+C\delta^2)|x-y|.
\end{align}
Moreover, by \eqref{approximate}, we know 
\begin{align*}
|\tilde{f}(x)-A(x)|=|P(x+x_0)-f(x+x_0)|\le C\delta (1+|x_0+x|^2)\le C\delta(1+|x_0|^2+|x|^2).
\end{align*}
Noting that $P$ is an orthogonal immersion, by Lemma \ref{compactness}, we know 
\begin{align*}
|x_0|=|P(x_0)|\le |P(x_0)+v|+|v|\le |P(0)+v|+|v|=\frac{2}{\text{diam}(\Sigma)}\le 14,
\end{align*}
and hence 
\begin{align}\label{translated approximation}
|\tilde{f}(x)-A(x)|\le C\delta(15+|x|^2)\le C\delta(1+|x|^2).
\end{align}
Now, define $\Phi:S_r\to \Sigma$ by 
$$\Phi(q)=\pi^{-1}(\tilde{f}(A^{-1}(\pi(q)))),\forall q\in S_r.$$
For $q\in S_r$, we know  $\tilde{q}=\pi(q)=\frac{q}{|q|^2}\in A(\mathbb{R}^2)$ and hence $x=A^{-1}(\tilde{q})\in \mathbb{R}^2$. So, we know $\tilde{p}=\tilde{f}(x)\in \tilde{\Sigma}$ and hence $p=\pi^{-1}(\tilde{q})=\frac{\tilde{q}}{|\tilde{q}|^2}\in \Sigma$. Noting that each of the maps are conformal when restricted to their domain, we get a  conformal parameterization from $S^2$ to $\Sigma$. 
Next, we are going to estimate the $L^\infty$ norm and conformal factor of $\Phi.$  Take notations as above, then 
\begin{align*}
|\Phi(q)-q|&=|\pi^{-1}(\tilde{f}(x))-\pi^{-1}(A(x))|\\
&=\big|\frac{\tilde{f}(x)}{|\tilde{f}(x)|^2}-\frac{A(x)}{|A(x)|^2}\big|\\
&\le \frac{|\tilde{f}(x)-A(x)|}{|\tilde{f}(x)|^2}+\frac{||A(x)|^2-|\tilde{f}(x)|^2|}{|\tilde{f}(x)|^2|A(x)|}\\
&\le \frac{2|\tilde{f}(x)-A(x)|}{|\tilde{f}(x)|^2}+\frac{|\tilde{f}(x)-A(x)|}{|\tilde{f}(x)| |A(x)|}.
\end{align*}
Noting that $A(x)\in P(\mathbb{R}^2)+v$ and \eqref{away from zero}  implies  $|A(x)|\ge \frac{1}{2d}$, by  $|A(x)-A(0)|=|P(x)|=|x|$, $|A(0)|=|v'|\le |v|\le \frac{1}{d}\le 7$ , we know 
\begin{align*}
\max\{\frac{1}{2d},|x|-7\}\le |A(x)|\le |x|+7.
\end{align*}
Similarly,  $\tilde{f}(x)\in \tilde{\Sigma}$ and \eqref{distance lower bound} implies $|\tilde{f}(x)|\ge \frac{1}{d}$, by 
$
|\tilde{f}(0)-A(0)|+|A(0)|\le C\delta+7\le 8
$ and \eqref{tilde bi-Lipschitz}, 
we know 
\begin{align*}
\max\{\frac{1}{d},(1-C\delta)|x|-8\}\le |\tilde{f}(x)|\le (1+C\delta^2)|x|+8
\end{align*}
So, we know $|\tilde{f}(x)|\ge \frac{1}{C}(1+|x|)$ and $|A(x)|\ge \frac{1}{C}(1+|x|)$, which combining with \eqref{translated approximation} implies
\begin{align} \label{pointwise estimate}
|\Phi(q)-q|&\le  \frac{2|\tilde{f}(x)-A(x)|}{|\tilde{f}(x)|^2}+\frac{|\tilde{f}(x)-A(x)|}{|\tilde{f}(x)| |A(x)|}\le \frac{C\delta(1+|x|^2)}{\frac{1}{C}(1+|x|)^2}\le C\delta.
\end{align}
Moreover, since $\Phi: S_r\to \Sigma$ is conformal, we can set $d\Phi\otimes d\Phi=e^{2v}g_{S_r}$ and calculate by \eqref{differential of inversion} and the orthogonality of $dA=P$ to get 
\begin{align*}
e^{2v(q)}=\frac{e^{2w(x)}}{|\tilde{p}|^4|q|^4}=\frac{|A(x)|^4e^{2w(x)}}{|\tilde{f}(x)|^4}. 
\end{align*}
Noting that $|A(x)|^2=|P(x+x_0)+v|^2=|P(x)+v'|^2$, by \eqref{perpendicular}, we know 
$$|A(x)|=|P(x)|^2+|v'|^2=|x|^2+|v'|^2.$$
On the other hand,  by $v'=P(x_0)+v$, we know $\tilde{f}(x)=f(x+x_0)+v=f(x+x_0)-P(x_0)+v'$ and hence 
\begin{align}\label{tilde f estimate}
|\tilde{f}(x)|^2=\underbrace{|f(x+x_0)-P(x_0)|^2}_{I}+\underbrace{2\langle f(x+x_0)-P(x_0),v'\rangle}_{II} +|v'|^2.
\end{align}
Again \eqref{perpendicular}, \eqref{orthogonal approximation} and $\frac{1}{C}\le \frac{1}{2d}\le |v'|\le 7$, we know 
\begin{align}\label{II estimate}
|II|=|\langle f(x+x_0)- P(x_0+x),v'\rangle|\le C\delta(1+|x|^2) |v'|\le C\delta(|v'|^2+|x|^2).
\end{align}
Noting 
\begin{align*}
I&=|f(x+x_0)-P(x_0)|^2\\
&=|f(x+x_0)-f(x_0)|^2+|f(x_0)-P(x_0)|^2+2\langle f(x+x_0)-f(x_0),f(x_0)-P(x_0)\rangle,
\end{align*}
where 
\begin{align*}
(1-C\delta)|x|^2\le |f(x+x_0)-f(x_0)|^2\le (1+C\delta^2)|x|^2,
\end{align*}
\begin{align*}
|f(x_0)-P(x_0)|^2\le [C\delta (1+|x_0|)]^2\le (15 C\delta)^2\le C\delta^2,
\end{align*}
and 
\begin{align*}
|2\langle f(x+x_0)-f(x_0),f(x_0)-P(x_0)\rangle|\le 2(1+C\delta^2)|x|\cdot C\delta(1+|x_0|)\le C\delta(|v'|^2+|x|^2),
\end{align*}
we know that 
\begin{align}\label{I estimate}
I\le (1+C\delta^2)|x|^2+C\delta^2+C\delta(|v'|^2+|x|^2)\le (1+C\delta)|x|^2+C\delta|v'|^2 ,
\end{align}
and similarly there holds
\begin{align}\label{I lowerbound estimate}
I\ge (1-C\delta)|x|^2-C\delta |v'|^2. 
\end{align}
Substituting \eqref{I estimate}, \eqref{I lowerbound estimate} and \eqref{II estimate} into \eqref{tilde f estimate}, we get 
\begin{align*}
(1-C\delta )(|x|^2+|v'|^2)\le |\tilde{f}(x)|^2\le (1+C\delta)(|x|^2+|v'^2|).
\end{align*}
So, by $|A(x)|^2=|x|^2+|v'|^2$, we know 
\begin{align}\label{inverted C0 estimate}
1-C\delta\le \frac{|\tilde{f}(x)|^2}{|A(x)|^2}\le 1+C\delta,
\end{align}
and hence by \eqref{intrinsic estimate}, we get 
\begin{align}\label{conformal factor estimate}
|v(q)|=\frac{1}{2}|\log{\frac{|\tilde{f}(x)|^2}{|A(x)|^2}}|+|w(x)|\le C\delta+C\delta^2\le C\delta.
\end{align}
This implies the intrinsic distance between $S_r$ and $\Sigma$ is bi-Lipschitz, i.e,, 
$$1-C\delta\le \frac{d_{\Sigma}(\Phi(q_1),\Phi(q_1))}{d_{S_r}(q_1,q_2)}\le 1+C\delta.$$
Next, we show the restriction distance of $d_{\mathbb{R}^n}$ on $S_r$ and $\Sigma$ are also bi-Lipschitz under the map $\Phi$. More precisely, for any $q_1,q_2\in S_r$, letting $x_i=A^{-1}(\pi(q_i))$ and $p_i=\Phi(q_i)$ for $i=1,2$. Then, we know $p_i=\frac{\tilde{f}(x_i)}{|\tilde{f}(x_i)|^2}$ and $q_i=\frac{A(x_i)}{|A(x_i)|^2}$. So, we know 
\begin{align*}
|q_1-q_2|^2&=|\frac{A(x_1)}{|A(x_1)|^2}-\frac{A(x_2)}{|A(x_2)|^2}|^2=\frac{1}{|A(x_1)|^2}+\frac{1}{|A(x_2)|^2}-\frac{2\langle A(x_1), A(x_2)\rangle }{|A(x_1)|^2|A(x_2)|^2}\\
          &=\frac{|A(x_1)-A(x_2)|^2}{|A(x_1)|^2|A(x_2)|^2}=\frac{|x_1-x_2|^2}{|A(x_1)|^2|A(x_2)|^2}.
\end{align*}
Similarly, there holds 
\begin{align*}
|p_1-p_2|^2=\frac{|\tilde{f}(x_1)-\tilde{f}(x_2)|^2}{|\tilde{f}(x_1)|^2|\tilde{f}(x_2)|^2}.
\end{align*}
Thus by \eqref{tilde bi-Lipschitz} and \eqref{inverted C0 estimate}, we know 
\begin{align*}
\frac{|p_1-p_2|^2}{|q_1-q_2|^2}=\frac{|A(x_1)|^2|A(x_2)|^2}{|\tilde{f}(x_1)|^2|\tilde{f}(x_2)|^2}\frac{|\tilde{f}(x_1)-\tilde{f}(x_2)|^2}{|x_1-x_2|^2}\in [(1-C\delta)^2(1-C\delta)^2,(1+C\delta)^2(1+C\delta^2)^2],
\end{align*}
which means 
\begin{align}\label{extrinsic bi-Lipschitz}
1-C\delta\le \frac{|\Phi(q_1)-\Phi(q_2)|}{|q_1-q_2|}\le 1+C\delta.
\end{align}
Finally, by $\mu(R^n)=\int_{\Sigma}\Theta(\mu,x)d\mathcal{H}^2=4\pi$, $\mathcal{H}^2(\Sigma)=\int_{S_r}e^{2v}d\mathcal{H}^2$ and noting Lemma  \ref{Willmore lower bound} implies $1\le \Theta(\mu,x)\le 1+\delta^2$, we know that
\begin{align*}
\frac{4\pi}{1+\delta^2}\le \mathcal{H}^2(\Sigma)\le 4\pi \quad \quad \text{ and } \quad \quad e^{-C\delta} 4\pi r^2\le \mathcal{H}^2(\Sigma)\le e^{C\delta}4\pi r^2, 
\end{align*}
which implies 
\begin{align*}
1-C\delta\le r\le 1+C\delta.
\end{align*}
So, after scaling and translation, we can replace $S_r$ by $\mathbb{S}^2=\frac{1}{r}(S_r-\frac{v'}{2|v'|^2})-\vec{c}$(where $\vec{c}$ is the mass center of $\frac{1}{r}(S_r-\frac{v'}{2|v'|^2})$ in $\mathbb{R}^n$), replace $\Phi$ by $\hat{\Phi}(\cdot)=\Phi(r\cdot+\frac{v'}{2|v'|^2})-\frac{v'}{2|v'|^2}-\vec{c}$, and translate $V=\underline{v}(\Sigma,\theta)$ by $\tau(\cdot)=\cdot-\frac{v'}{2|v'|^2}-\vec{c}$, then the estimates \eqref{pointwise estimate}, \eqref{conformal factor estimate} and \eqref{extrinsic bi-Lipschitz} still hold for $\hat{\Phi}:\mathbb{S}^2\to \Sigma-\frac{v'}{2|v'|^2}-\vec{c}$. 

\end{proof}

Now, we can apply the argument of Lamm and Sch\"atzle\cite{LS-2014} to get the $W^{2,2}$ estimate. For completeness, we write down the details. Instead of using the fact $V$ is a rectifiable varifold  underlying an integral current, we use a topological argument to show the projection map $\text{Proj}:\Sigma\to S^2$ is surjective. 

\begin{corollary}
Under the conclusion of Theorem \ref{partial Quantitative ridigity}, there holds 
\begin{align*}
\int_{\mathbb{S}^2}|\Delta\Phi+2\Phi|^2d\mathcal{H}^2\le C\delta^2
\end{align*}
and 
\begin{align}\label{W^{2,2} estimate}
\|\Phi-id_{\mathbb{S}^2}\|_{W^{2,2}(\mathbb{S}^2)}\le C\delta.
\end{align}
\end{corollary}
\begin{proof}
Noting that the mean curvature equation  $\Delta \Phi=H\circ \Phi e^{2v}$ implies 
\begin{align}\label{reduction 1}
\|\Delta\Phi+2\Phi\|_{L^2(\mathbb{S}^2)}&\le \|H\circ \Phi+2\Phi\|_{L^2(\mathcal{S}^2)}+\big(\int_{\mathbb{S}^2}|H\circ \Phi|^2|e^{2v}-1|^2d\mathcal{H}^2\big)^{\frac{1}{2}}\nonumber \\
&\le \int_{\Sigma}|H|^2 d\mu)^{\frac{1}{2}}\sup{\frac{|e^{2v}-1|}{e^v\sqrt{\Theta(\mu,x)}}}+\|H(p)+2p\|_{L^2(d\mu)}\sup{\frac{1}{e^v\sqrt{\Theta(\mu,x)}}}\nonumber\\
&\le C(\delta+\|H(p)+2p\|_{L^2(d\mu)}),
\end{align}
we only need to estimate 
\begin{align}\label{reduction 2}
\int_{\Sigma}|H(p)+2p|^2d\mu&=\int_{\Sigma}|H|^2d\mu+4\int_\Sigma |p|^2d\mu+4\int_{\Sigma}\langle p, H(p)\rangle d\mu\nonumber\\
&\le 16\pi(1+\delta^2)+4\int_{\Sigma}|p|^2-4\int_{\Sigma}\underbrace{div^{\Sigma}p}_{=2}d\mu\nonumber\\
&=16\pi(1+\delta^2)+4\int_{\Sigma}|p|^2d\mu-32\pi\nonumber\\
&\le -16\pi+4\int_{\Sigma}|p|^2d\mu+16\pi\delta^2.
\end{align}
Now, follow the argument as in \cite{LS-2014}, denote $\mathbb{R}^n=V_1\oplus V_2$, where $V_1=\{\lambda q| q\in \mathbb{S}^2\}$ and $V_2=V_1^{\bot}$, and consider $\text{Proj}:B_{C\delta}(\mathbb{S}^2)\to \mathbb{S}^2$ defined by $\text{Proj}(u,v)=\frac{u}{|u|}$, where $(u,v)\in B_{C\delta}(\mathbb{S}^2)\cap (V_1\oplus V_2)$.  By \eqref{pointwise estimate}, we know $\Sigma\subset B_{C\delta}(\mathbb{S}^2)$ and $\varphi=\text{Proj}\circ \Phi: \mathbb{S}^2\to \Sigma\to \mathbb{S}^2$ satisfies 
\begin{align*}
|\varphi(q)-q|\le |\pi(\Phi(q))-\Phi(q)|+|\Phi(q)-q|=d(\Phi(q),\mathbb{S}^2)+|\Phi(q)-q|\le 2|\Phi(q)-q|\le C\delta.
\end{align*}
This implies $h(t,q)=t\varphi(q)+(1-t)q\neq 0$ for any $t\in [0,1]$ and $q\in \mathbb{S}^2$, and hence $H:[0,1]\times \mathbb{S}^2\to \mathbb{S}^2$ defined by $H(t,q)=\frac{h(t,q)}{|h(t,q)|}$ is a homotopy joining $\varphi:\mathbb{S}^2\to \mathbb{S}^2$ and $id_{\mathbb{S}^2}$.  This further implies $\text{Proj}:\Sigma\to \mathbb{S}^2$ is surjective and hence $\mathcal{H}^2(\text{Proj}(\Sigma))=4\pi.$  Moreover, since $D\text{Proj}=\frac{|u|^2I_3-u\otimes u}{|u|^3}$ is a matrix whose nonzero eigenvalue equals to $\frac{1}{|u|}$ with multiplicity two, we know $|D\text{Proj}|^2=\frac{2}{|u|^2}$. Thus by the co-area formula and noting $\Theta(\mu,x)\ge 1$, we know  
\begin{align*}
4\pi&=\mathcal{H}^2(\text{Proj}(\Sigma))\le \int_{\Sigma}J_{\Sigma}\text{Proj}(p)d\mathcal{H}^2(p))\\
&\le \int_{\Sigma}\frac{1}{2}|D\text{Proj}|^2(p)\Theta(\mu,p)d\mathcal{H}^2(p)\\
&\le \int_{\Sigma}\frac{1}{|u|^2}d\mu\\
&=\int_{\Sigma}2-|u|^2+\frac{(|u|^2-1)^2}{|u|^2}d\mu\\
&=\int_{\Sigma}2-|p|^2+|v|^2+\frac{(|u|^2-1)^2}{|u|^2}d\mu,
\end{align*}
and hence 
\begin{align*}
\int_{\Sigma}|p|^2d\mu\le 4\pi+4\pi\sup_{q=(u,v)\in \Sigma}(|v|^2+\frac{(|u|^2-1)^2}{|u|^2})
\end{align*}
Since $\Phi:\Sigma\to \mathbb{S}^2$ is surjective, we also know that there exists $q\in \Sigma$ such that $\Phi(q)=p$ and hence 
$$|\pi(p)-p|=d(\Phi(q),\mathbb{S}^2)\le |\Phi(q)-q|\le C\delta,$$
which means
\begin{align*}
C\delta^2\ge |\pi(p)-p|^2=|\frac{u}{|u|}-u|^2+|v|^2=(1-|u|)^2+|v|^2.
\end{align*}
So, we know $||u|-1|\le\delta $, $|v|^2\le C\delta^2$ and hence 
$$\int_{\Sigma}|p|^2d\mu\le 4\pi+ 4\pi[C\delta^2+\frac{(1+\delta)^2}{1-\delta}\delta^2]\le 4\pi+ C\delta^2.$$
Substitute this into \eqref{reduction 2} and \eqref{reduction 1}, we get 
\begin{align*}
\int_{\Sigma}|H(p)+2p|^2d\mu\le C\delta^2,
\end{align*}
and hence
\begin{align*}
\int_{\mathbb{S}^2}|\Delta \Phi+2\Phi|^2d\mathcal{H}^2\le C\delta^2.
\end{align*}
Finally, for the round sphere $\mathbb{S}^2$ constructed in Theorem \ref{partial Quantitative ridigity}, denoting $id_{\mathbb{S}^2}(q)=q$, then  $\Delta id_{\mathbb{S}^2}+2id_{\mathbb{S}^2}=0$ and (\ref{C0 estimate}) implies $\|\Phi-id_{\mathbb{S}^2}\|_{L^2(\mathbb{S}^2)}\le C\delta$. So, by standard elliptic estimate, we get 
$$\|\Phi-id_{\mathbb{S}^2}\|_{W^{2,2}(\mathbb{S}^2)}\le C(\|\Delta \Phi+2\Phi\|_{L^2(\mathbb{S}^2)}+\|\Phi-id_{\mathbb{S}2}\|_{L^2(\mathbb{S}^2)})\le C\delta.$$
This completes the proof. 
\end{proof}

\section{Sharp constant}\label{sec:sharp constant}
We end this paper by a discussion of the sharp constant  in $\mathbb{R}^3$.  Let us first recall the conception of $W^{2,2}$ conformal immersion\cite{KL-2012}
\begin{defi} We call a varifold $V=\underline{v}(\Sigma,\theta)$ a $W^{2,2}$-conformal immersion if there exists a Riemann surface $\Sigma_0$ and a conformal parameterization $f\in \Sigma_0 \to \Sigma$ such that $f\in W^{2,2}(\Sigma_0,\mathbb{R}^n)$ and $df\otimes df=e^{2u}g_0$ satisfies $\|u\|_{L^\infty}<+\infty$, where $g_0$ is a fixed smooth metric on $\Sigma_0$.  
\end{defi}
\begin{theorem}\label{sharp constant}
Assume $V=\underline{v}(\Sigma,\theta)$ is an integral varifold in $\mathbb{R}^3$ with finite volume and $\mathcal{W}(V)<2\pi^2$. Then, 
\begin{enumerate}[1)]
\item either $\Theta(\mu,x)<\frac{3}{2}$ for any $x\in \Sigma$  and $V$ is a $W^{2,2}$ conformal embedded sphere in $\mathbb{R}^3$,
\item or there exists $x_0\in \Sigma$ such that  $\Theta(\mu,x_0)\ge \frac{3}{2}$. In this case, we have  $\mathcal{W}(V)\ge 6\pi$. Moreover, the equality $\mathcal{W}(V)=6\pi$ holds if and only if $\Sigma$ is the double bubble up to a conformal transformation. 
\end{enumerate} 
\end{theorem}
\begin{proof}
 Since $W(V)+\mu_V(\mathbb{R}^n)<+\infty$, by the monotonicity formula and Allard compactness theorem, we know for every $p\in \Sigma$, any tangent cone $V_p=\underline {v}(T_p,\theta_p)$ of  $V$ at $p$ is an integral stationary cone  with $T_p=C(S_p), \theta_p(\lambda \omega)=\theta_p(\omega)$ for any $\omega\in S_p:=T_p\cap \mathbb{S}^{2}$, where $\underline{v}(S_p,\theta_p(\omega))$ is an integral stationary 1-varifold in $\mathbb{S}^2$.  Moreover,  the monotonicity formula implies 
\begin{align*}
\Theta(V_p,x)\le \Theta(V_p,\infty)=\Theta(V_p,0)=\Theta(\mu,p), \forall x\in T_p, 
\end{align*}
Noting that for any $\omega\in T_p\cap \mathbb{S}^2$, there holds
\begin{align*}
T_\omega (V_p)=\mathbb{R}\oplus T_\omega (S_p)
\end{align*}
and hence 
\begin{align*}
\Theta(V_p,\omega)=\Theta(T_\omega V_p,0)=\Theta(T_\omega S_p,0)=\Theta(S_p,\omega).
\end{align*}
By the structure of stationary $1$-varifold\cite{AA-1976}, we know when $\Theta(\mu,p)<\frac{3}{2}$, there holds  $\Theta(S_p,\omega)=1$ for any $\omega\in S_p$ and  hence $S_p\subset \mathbb{S}^2$ is a great circle.  This further implies $T_p=C(S_p)=\mathbb{R}^2$  and $\theta_p\equiv 1$ by the constant theorem\cite[Theorem 41.1]{LS83}.

1). So, in the case $\Theta(\mu,x)<\frac{3}{2}$ for any $x\in \Sigma$,  we know $\Theta(\mu,x)=\Theta(T_x,0)\equiv 1$ for any $x\in \Sigma$.  Thus by Theorem \ref{2022b-main}, we know $V$ is a $W^{2,2}$ immersion in $\mathbb{R}^3$,i.e., there 
 is a compact(by Lemma \ref{compactness}) Riemann surface $\Sigma_0$ and a conformal parameterization $f:\Sigma_0\to \Sigma$ such $f\in W^{2,2}(\Sigma_0,\mathbb{R}^3)$ with $df\otimes df=e^{2u}g_0$ and $\|u\|_{L^\infty}\le C$.  This implies $\Sigma$ can be approximated by smooth oriented immersions. So, by Marques and Neves' theorem (Willmore conjecture)\cite{MN-2014}, we know $\Sigma_0=\mathbb{S}^2$ and the map $f:\mathbb{S}^2\to \Sigma\subset \mathbb{R}^3$ is embedded.  

 2).  In the case $\Theta(\mu,x_0)\ge \frac{3}{2}$ for some $x_0\in\Sigma$, by the monotonicity formula (Lemma \ref{Willmore lower bound}), we know $W(V)\ge 6\pi$. Moreover,  $W(V)=6\pi$ if and only if  $\Theta(\mu,x_0)=\frac{3}{2}$ and
 \begin{align*}
 \int_{\mathbb{R}^n}|\frac{H}{4}+\frac{\nabla^{\bot}r}{r}|^2d\mu=0,
 \end{align*}
 where $r=|\cdot -x_0|$.  Letting $f_0(x)=\frac{x-x_0}{|x-x_0|^2}+x_0$ and $\tilde{V}=f_{0\sharp}V=\underline{v}(\tilde{\Sigma},\tilde{\theta})$, then, by the inversion Lemma \ref{inverted lemma}, we know $\tilde{V}$ is an integral stationary varifold with 
 \begin{align*}
 \Theta(\tilde{\mu},\infty)=\Theta(\mu,x_0)=\frac{3}{2}. 
 \end{align*}
 Now, we consider $q\in \tilde{\Sigma}$. If $\Theta(\tilde{\mu},q)<\frac{3}{2}$ for any $q\in \tilde{\Sigma}$, then the same argument as the first case implies $\tilde{\Sigma}$ is a smooth minimal surface in $\mathbb{R}^3$ with $\Theta(\tilde{\Sigma},\infty)<\frac{3}{2}$, which further implies $\tilde{\Sigma}$ is a plane and hence $\Sigma$ is a round sphere with multiplicity one (by the constant theorem). This contradicts to the fact $\Theta(\mu,x_0)=\frac{3}{2}$. So, we know there exists $q\in \tilde{\Sigma}$ such that 
 \begin{align*}
 \Theta(\tilde{\mu},q)=\Theta(\tilde{\mu},\infty)=\frac{3}{2}. 
 \end{align*}
 The monotonicity formula again implies $\tilde{V}=\underline{v}(C(S),\tilde{\theta})$ is a stationary cone with respect to $q$. Without loss of generality, assume $q=0$, then  $V_1=\underline{v}(S,\tilde{\theta}(\omega))$ is an integral stationary 1-varifold in $S^2$.  Moreover, the monotonicity formula implies 
 \begin{align*}
 L:=\{x\in \tilde{\Sigma}| \Theta(\tilde{\mu},x)\ge \Theta(\tilde{\mu},q)\}
 \end{align*}
 is a linear space and $\tilde{V}$ is a splitting cone with respect to $L$\cite{LS96}. Again by the structure of stationary 1-varifolds\cite{AA-1976} and $\Theta(\tilde{\mu},q)=\frac{3}{2}>1$, we know $S$ is not smooth and there exists $\omega\in S$ such that 
 \begin{align*}
 \Theta(\tilde{\mu},\omega)=\Theta(V_1,\omega)\ge \frac{3}{2}. 
 \end{align*}
 Thus we know $\text{dim} L\ge 1$ and 
 \begin{align*}
 \tilde{\Sigma}=\mathbb{R}\times C(S_1), 
 \end{align*}
 where $C(S_1)\subset \mathbb{R}^2$ is a stationary cone with $\Theta(C(S_1),\infty)=\frac{3}{2}$. As a result, $S_1$ is the $Y$-singularity cone  and hence $\tilde{\Sigma}=\mathbb{R}\times Y$, which implies $\Sigma=f_0^{-1}(\Sigma)$ is the double bubble up to a conformal transformation. 
\end{proof}
\begin{remark}  In the above theorem, Case $2)$ also holds for $n\ge 3$ by the same argument, and Case $1)$ needs $n=3$ since Marques and Neves' theorem is applied in the proof.  For $n\ge 3$, even it is assumed $w(V)<6\pi$=the Willmore energy of Veronese embedding of   $\mathbb{RP}^2$ in $\mathbb{R}^4$, it is not known whether $V$ is a topological sphere. It is exactly a conjecture by Kusner which says a surface in $\mathbb{R}^4$ with Willmore energy smaller than $6\pi$ must be a sphere\cite{K96}.  
\end{remark}
\begin{corollary}\label{corollary} When $n=3$, the conclusion of Theorem \ref{main theorem} holds for $W(V)=4\pi+\delta^2<6\pi$ and $6\pi$ is the sharp constant.
\end{corollary}
\begin{proof}
Let $\delta_0(3)$ be the constant occurs in Theorem \ref{main theorem}. 
When $\delta\in (0,\delta_0(3))$, the conclusion holds by Theorem \ref{main theorem}. When $\delta^2\in[\delta_0^2,2\pi)$,
since $V=\underline{v}(\Sigma,\theta)$ is an integral varifold in $\mathbb{R}^3$ with finite volume and $\mathcal{W}(V)<6\pi$, by Theorem \ref{sharp constant}, we know  $V$ is a $W^{2,2}$ conformal embedded sphere. Thus by \cite[Proposition 3.1]{LS-2014}, we know there exists a conformal parameterization $f:\mathbb{S}^2\to \Sigma$ with pull back metric $g=df\otimes df=e^{2u}g_{\mathbb{S}^2}$ and $\|u\|_{L^\infty(\mathbb{S}^2)}\le C.$ 
Now, by Lemma \ref{compactness} we know $\|f\|_{L^\infty}\le C$, which combining with the elliptic estimates of the mean curvature equation gives $\|f\|_{W^{2,2}}\le C$. As a result, there holds 
\begin{align*}
\|u\|_{L^{\infty}}+\|f-id\|_{W^{2,2}}\le C\delta_0\le C\delta. 
\end{align*}
Finally, the double bubble example in case $2)$ of Theorem \ref{sharp constant} implies $6\pi$ is the sharp constant. 
\end{proof}
	
\end{document}